\documentclass[12pt,reqno,a4paper]{amsart}

\usepackage{hyperref}

\usepackage{comment}

\usepackage{graphicx}
\usepackage{epstopdf}

\usepackage{amsmath,amsfonts,amsthm,mathrsfs,amssymb,cite}
\usepackage[usenames]{color}
\usepackage{cite}
\usepackage{subfigure}

\newtheorem{thm}{Theorem}[section]

\newtheorem{lem}{Lemma}[section]

\theoremstyle{definition}

\theoremstyle{remark}

\newtheorem{rem}{Remark}[section]
\numberwithin{equation}{section}

\newcommand{\curl}{\mathbf{curl}}
\newcommand{\diver}{\mathrm{div}}

\newcommand{\pare}[1]{\left(#1\right)}

\theoremstyle{remark}

\def\bold{\boldsymbol}

 \def\p{\partial} \def\nb{\nonumber}
\def\to{\rightarrow}
\def\Om{\Omega}  \def\om{\omega}
    
\def\l{\label}  
   
\def\m{\mbox} \def\t{\wedge}  \def\lam{\lambda}

\def\bb{\begin{equation}} \def\ee{\end{equation}}

\def\beqn{\begin{eqnarray}}  \def\eqn{\end{eqnarray}}

\def\beqnx{\begin{eqnarray*}} \def\eqnx{\end{eqnarray*}}

\title[Decoupling Elastic Waves]{Revisiting the Decoupling of Elastic Waves From a Weak Formulation Perspective }

\author{Junjiang Lai}
\address{Department of Mathematics, Minjiang University, Fuzhou 350108, China }
\email{laijunjiang@163.com}

\author{Hongyu Liu}
\address{Department of Mathematics, Hong Kong Baptist University, Kowloon, Hong Kong, China}
\email{hongyu.liuip@gmail.com}

\author{Jingni Xiao}
\address{Department of Mathematics, Hong Kong Baptist University, Kowloon, Hong Kong, China}
\email{xiaojn@live.com}

\author{Yifeng Xu}
\address{Department of Mathematics and Scientific Computing Key Laboratory of Shanghai Universities, Shanghai Normal University, Shanghai 200234, China.}
\email{yfxu@shnu.edu.cn, yfxuma@aliyun.com}

\begin{document}

\begin{abstract}

Elastic scattering governed by the Lam\'e system associated with the third-type or fourth-type boundary condition is considered. It was shown in \cite{liuxiao} by two of the authors that under suitable geometric conditions on the boundary surface of the elastic inclusion, the longitudinal and shear waves can be decoupled. The decoupling result in \cite{liuxiao} was derived based on analyzing the local boundary behaviours of the elastic fields. In this article, we provide a different argument from a variational perspective in proving the decoupling result.


\end{abstract}

\maketitle

\section{Introduction}

Let $\Om\subset\mathbb{R}^3$ be a bounded simply-connected domain with a connected {\color{black}piecewise} $C^{2,1}$-smooth boundary $\p\Om$. An isotropic elastic medium, characterized by the Lam\'e constants $\lambda> 0$ and $\mu>0$, occupies $\Om$. Given a source term $\bold{f}\in \bold{L}^2(\Om)$,
we consider the linearized elasticity equation
\begin{equation}\label{lame-sys}
-\Delta^\ast\bold{u}-\om^2\bold{u}=\bold{f}\quad\m{in}~\Om,
\end{equation}
where $\bold{u}=[u_j(\boldsymbol{x})]_{j=1}^3$ and $\om>0$ denote the displacement field and the angular wavenumber, respectively and the operator $\Delta^*$ is defined by
\begin{align}
\Delta^*\bold{u}=&\mu\Delta \bold{u}+(\lambda+\mu)\boldsymbol{\nabla}(\boldsymbol{\nabla}\cdot \bold{u})\nonumber\\
=&-\mu\boldsymbol{\nabla}\wedge(\boldsymbol{\nabla}\wedge \bold{u})+(\lambda+2\mu)\boldsymbol{\nabla}(\boldsymbol{\nabla}\cdot \bold{u}).\label{atwo}
\end{align}

The Lam\'e system \eqref{lame-sys} is complemented by one of the following four kinds of boundary conditions{\color{black} \cite{Knops1971,Kupradz1979}.} The first kind boundary condition is given by
\begin{equation}\l{1st_bc}
\bold{u}=\bold{0}\quad\m{on}~\p\Om.
\end{equation}
The second kind one reads
\bb\label{2nd-bc}
\bold{Tu}=\bold{0}\quad\m{on}~\p\Om.
\ee
where $\boldsymbol{T}$ is the traction operator on $\partial \Omega$ defined by
\begin{equation*}\label{eq:tractionoperator}
\begin{split}
\boldsymbol{T}{\boldsymbol{u}}
:=&\lambda (\bold{\nabla}\cdot \bold{u})\bold{\nu}
+\mu (\bold{\nabla}\bold{u}+\bold{\nabla}^T\bold{u})\bold{\nu}  \\=&2\mu\partial_{\boldsymbol{\nu}} {\bold{u}}+\lambda\,{\boldsymbol{\nu}}\boldsymbol{\nabla}\cdot {\bold{u}}+\mu\,{\boldsymbol{\nu}}\wedge(\boldsymbol{\nabla}\wedge {\bold{u}}),
\end{split}	
\end{equation*}
with the superscript $\mathrm{T}$ denoting the transpose of a matrix, ${\boldsymbol{\nu}}\in\mathbb{S}^2$ the outward unit normal vector to $\partial \Omega$, and $\partial _{\boldsymbol{\nu}}$ the boundary differential operator defined as
\[
\partial_{\boldsymbol{\nu}} {\bold{u}}:=[{\boldsymbol{\nu}}\cdot\boldsymbol{\nabla} {u}_{j}]_{j=1}^3.
\]
The boundary conditions of the third and fourth kind are given respectively, by
\bb\label{3rd-bc}
\bold{\nu}\cdot\bold{u}=0,\quad\bold{\nu}\t\bold{Tu}=\bold{0}\quad\mbox{on}~\p\Om,
\ee
and
\bb\label{4th-bc}
\bold{\nu}\t\bold{u}=\bold{0},\quad\bold{\nu}\cdot\bold{Tu}=0\quad\mbox{on}~\p\Om.
\ee

In this article, we are mainly concerned with the Lam\'e system \eqref{lame-sys} associated with the third \eqref{3rd-bc} or the fourth \eqref{4th-bc} boundary condition.
It is known that the elastic body waves can be decomposed into two parts: pressure wave and shear wave, which generally coexist and propagate simultaneously at different speeds and directions. Recently, in \cite{liuxiao} under {\color{black}certain geometric conditions} 
on the boundary of a impenetrable scatterer, a complete decoupling of the pressure wave and shear wave is established for the third and the fourth boundary conditions, respectively. Further, this allows reformulation of the Lam\'{e} system as a Helmholtz system and a Maxwell system. 

The aim of this article is to revisit the decoupling results in
\cite{liuxiao} by a variational approach. By the existing results
for boundary conditions of the third and the fourth kind in \cite{liuxiao}, we obtain two relevant variational formulations; see \eqref{vp_4th-bc} and \eqref{vp_3rd-bc}. By virtue of the Fredholm alternative theory, {\color{black} well-posednesses} 
of these two variational problems are proved; see Theorems \ref{thm:4thweak} and \ref{thm:3rdweakequ}.
We note that \eqref{vp_4th-bc} and \eqref{vp_3rd-bc} are well-posed without the geometric assumptions on the boundary surface.
{\color{black}{\color{black}Furthermore, } under a reasonable assumption on the regularity of the solution to the Lam\'{e} system, the variational formulations are reformulated as a Helmholtz equation and a Maxwell system; see Theorems \ref{thm:4thdecoup} and \ref{thm:3rddecoup}.  {\color{black} In addition, if certain geometric conditions are imposed on the boundary of the domain,} 
we are able to show that the variational formulations are equivalent to the original Lam\'{e} system.} That means, we retrieve the  decoupling results deduced from the Lam\'{e} system itself in \cite{liuxiao}, but from a weak formulation perspective.

To our knowledge, the two variational formulations \eqref{vp_4th-bc} and \eqref{vp_3rd-bc} are novel in the literature and are of significant use in numerical simulation of the Lam\'{e} system associated with \eqref{3rd-bc} and \eqref{4th-bc}, especially by the variational discretization techniques, e.g. finite element methods \cite{BrennerScott2008},\cite{Ciarlet1978},\cite{Jin2014},\cite{Monk2003}. Further, due to the decoupling results in Theorems \ref{thm:4thdecoup} and \ref{thm:3rddecoup}, real computations and the related numerical analysis may be performed separately on a Helmholtz equation and a Maxwell system.


The rest of this paper is organized as follows. Some preliminaries required in later analysis is provided in section 2. In section 3, we present variational formulations for {\color{black}the Lam\'{e} system} \eqref{lame-sys} associated with {\color{black}the third \eqref{3rd-bc} or the fourth \eqref{4th-bc} boundary conditions,} from which  relevant decoupling results are deduced. Throughout the paper, 
we shall use $C$, with or without subscript, for a generic constant independent of the function under consideration, and it may take a different value at each occurrence.

\section{Preliminaries}

\subsection{Sobolev spaces}

We first introduce some Sobolev spaces, which are required in the subsequent analysis.
{\color{black} We adopt the standard notation for $L^2$-based Sobolev spaces and use $(\cdot,\cdot)_{G}$ to denote the $L^{2}(G)$ scalar product on an open bounded domain $G\subset\mathbb{R}$ which is omitted when $G=\Omega$.} 
{\color{black}
	The spaces $\boldsymbol{H}(\curl;\Omega)$ and $\boldsymbol{H}(\diver;\Omega)$ are given, respectively, by  (cf. \cite{AmroucheBernardiDaugeEtAl1998})}
\[
    \boldsymbol{H}(\curl;\Omega)=\{\boldsymbol{v}\in\boldsymbol{L}^{2}(\Omega);\boldsymbol{\nabla}\wedge\boldsymbol{v}\in\boldsymbol{L}^{2}(\Omega)
    \},
\]
and
\[
    \boldsymbol{H}(\diver;\Omega)=\{\boldsymbol{v}\in\boldsymbol{L}^{2}(\Omega);\boldsymbol{\nabla}\cdot\boldsymbol{v}\in L^{2}(\Omega)
    \}.
\]
{\color{black}
They are equipped with the norms, respectively,
\[
\|\bold{v}\|_{\boldsymbol{H}(\curl;\Omega)}=\pare{\|\bold{v}\|^2_{\bold{L}^2(\Omega)}+\|\boldsymbol{\nabla}\wedge \bold{v}\|^2_{\bold{L}^2(\Omega)}}^{1/2},
\]
and
\[
\|\bold{v}\|_{\boldsymbol{H}(\diver;\Omega)}=\pare{\|\bold{v}\|^2_{\bold{L}^2(\Omega)}+\|\boldsymbol{\nabla}\cdot \bold{v}\|^2_{L^2(\Omega)}}^{1/2}.
\]}
{\color{black}Define}
\[
    \bold{X}:=\boldsymbol{H}(\curl;\Omega)\cap\boldsymbol{H}(\diver;\Omega),
\]
{\color{black}with the norm
\[
\|\bold{v}\|_{\bold{X}}=\pare{\|\bold{v}\|^2_{\bold{L}^2(\Omega)}+\|\boldsymbol{\nabla}\wedge \bold{v}\|^2_{\bold{L}^2(\Omega)}+\|\boldsymbol{\nabla}\cdot \bold{v}\|^2_{L^2(\Omega)}}^{1/2}.
\]}
{\color{black}The following subspaces of $\bold{X}$ are the main solution spaces used in this paper:}
\[\bold{X}_N:=\{\boldsymbol{v}\in\boldsymbol{X};
\boldsymbol{\nu}\t\boldsymbol{v}=\boldsymbol{0}~\mbox{on}~\partial\Omega\},
\quad\bold{X}_T:=\{\boldsymbol{v}\in\boldsymbol{X};
\boldsymbol{v}\cdot\boldsymbol{\nu}=\boldsymbol{0}~\mbox{on}~\partial\Omega\}.
\]

\begin{lem}[\cite{Weber1980}]\label{compact}
	The spaces $\bold{X}_N$ and $\bold{X}_T$ are compactly embedded into $\bold{L}^2(\Om)$ when $\Omega$ is a bounded and connected Lipschitz domain.
\end{lem}

\subsection{Surface gradient}
Let $\Gamma$ be a regular piece of $\partial \Omega$ and be represented in the following parametric form,
\begin{equation}\label{eq:pf1}
\boldsymbol{x}(\boldsymbol{u})=\big({{x}}_1(u_1,u_2), x_2(u_1,u_2), x_3(u_1,u_2)\big)^{\mathrm{T}},\ \ \boldsymbol{u}=(u_1,u_2)^{\mathrm{T}}\in\mathbb{R}^2,
\end{equation}
 Denote by $g=(g_{jk})_{j,k=1}^2$ the first fundamental matrix of differential geometry for $\Gamma$ with
\begin{equation*}\label{eq:fdmm}
g_{jk}:=\frac{\partial \boldsymbol{x}}{\partial u_j}\cdot\frac{\partial \boldsymbol{x}}{\partial u_k},\ \ \ j, k=1,2.
\end{equation*}
Henceforth, we assume that the parametric form \eqref{eq:pf1} is such chosen that
\begin{equation*}\label{eq:pf2}
{\boldsymbol{\nu}}={\boldsymbol{\nu}}({\boldsymbol{x}})=\frac{1}{\sqrt{|g|}}\frac{\partial \boldsymbol{x}}{\partial u_1}\wedge\frac{\partial \boldsymbol{x}}{\partial u_2},\ \ \ |g|:=\mbox{det}(g),
\end{equation*}
is the outward unit normal vector of $\partial \Omega$ on $\Gamma$.
We denote by $\mathrm{Grad}_\Gamma$ the surface gradient operator on $\Gamma$ (cf. \cite{ColtonKress2012,Nedelec2001}).
Recall that for a sufficiently smooth function $\varphi$ defined in an open neighborhood of $\Gamma$ we have
\begin{equation*}\label{eq:sg1}
\mathrm{Grad}_\Gamma\varphi=\sum_{j,k=1}^2 g^{jk}\frac{\partial\varphi}{\partial u_j}\frac{\partial {\boldsymbol{x}}}{\partial u_k},
\end{equation*}
where
\begin{equation*}\label{eq:sg2}
(g^{jk})_{j,k=1}^2:=\big[(g_{jk})_{j,k=1}^2\big ]^{-1}.
\end{equation*}
{\color{black}We define}
\begin{equation}\label{eq:H}
\mathcal{S}(\bold{x})=\mathcal{S}_\Gamma({\boldsymbol{x}}):=\frac{1}{2}\sum_{l=1}^3\Big[\mathrm{Grad}_\Gamma{{\nu}}_l\Big]_l.
\end{equation}

\section{Decoupling under the weak formulations}
In this section, we present our main results of this paper.
An application of the Helmholtz decomposition to the source $\bold{f}\in\bold{L}^2(\Om)$ yields $\bold{f}=\bold{f}_p+\bold{f}_s$ with $\bold{\nabla}\t\bold{f}_p=\bold{0}$ and $\bold{\nabla}\cdot\bold{f}_s=0$.
As a consequence, the solution $\bold{u}$ to \eqref{lame-sys}
may be formally split as
\begin{equation}\label{hd_dispfield}
\bold{u}=\bold{u}_p+\bold{u}_s,
\end{equation}
where
\[
\bold{u}_p:=-\frac{1}{k_p^2}(\bold{\nabla}(\bold{\nabla}\cdot\bold{u})+\frac{1}{\lam+2\mu}\bold{f}_p)\quad\text{and}\quad
\bold{u}_s=\frac{1}{k_s^2}(\bold{\nabla}\t(\bold{\nabla}\t\bold{u})-\frac{1}{\mu}\bold{f}_s),
\]
with $k_p:=\om/\sqrt{\lam+2\mu}$ and $k_s:=\om/\sqrt{\mu}$.
The vector fields $\bold{u}_p$ and $\bold{u}_s$ are referred to as the pressure (longitudinal) and shear (transversal) parts of $\bold{u}$, respectively. It is easy to see
\[
    \bold{\nabla}\cdot\bold{u}=\bold{\nabla}\cdot\bold{u}_p,\quad
    \bold{\nabla}\t\bold{u}=\bold{\nabla}\t\bold{u}_s.
\]
In this sectin, we shall decouple $\bold{u}$ by deriving a Helmholtz equation and a Maxwell system with some suitable boundary conditions for $\bold{\nabla}\cdot\bold{u}$ and $\bold{\nabla}\t\bold{u}$, respectively.

\subsection{The fourth kind boundary condition}

We first deal with \eqref{lame-sys} with the fourth kind boundary condition \eqref{4th-bc}. Before introducing the weak formulation, we recall a decoupling result built in \cite{liuxiao}, for the elastic scattering problem governed by the Lam\'{e} system \eqref{lame-sys} with zero source. It should be noted that the relevant proof does not involve the right hand side of \eqref{lame-sys}.

\begin{lem}[\cite{liuxiao}]\label{lem:4thBCequiv}
	Let $\bold{u}\in\bold{H}^2(\Om)$ solve \eqref{lame-sys} with $\bold{f}\in\bold{L}^2(\Om)$. 
	If $\mathcal{S}(\bold{x})=0$ on a connected piece $\Gamma$ of $\partial\Omega$, then the fourth kind boundary condition \eqref{4th-bc} implies
	\begin{equation}\label{4th-bc_equiv}
	\bold{\nabla}\cdot\bold{u}=0\quad{\color{black}\mbox{on}~\Gamma.}
	\end{equation}
\end{lem}

Suppose $\bold{u}\in \bold{H}^2(\Om)$ solves \eqref{lame-sys} with the fourth kind boundary condition \eqref{4th-bc}. By integrating by parts we get for any $\bold{v}\in\bold{X}_N$
\begin{equation}\label{4th-bc_aux1}
(\bold{\nabla}\t\bold{\nabla}\t\bold{u},\bold{v})
=(\bold{\nabla}\t\bold{u},\bold{\nabla}\t\bold{v})
\end{equation}
and
\begin{equation}\label{4th-bc_aux2}
(\bold{\nabla}(\bold{\nabla}\cdot\bold{u}),\bold{v})
=(\bold{\nabla}\cdot\bold{u},\bold{\nabla}\cdot\bold{v})
-\langle\bold{\nabla}\cdot\bold{u},\bold{v}\cdot\bold{\nu}\rangle_{\p\Om}.
\end{equation}
where $\langle\cdot,\cdot\rangle_{\partial \Omega}$ is the duality pairing between $H^{1/2}(\p\Om)$ and $H^{-1/2}(\p\Om)$ since the normal trace space of $\bold{H}(\mathrm{div};\Om)$ is $H^{-1/2}(\p\Om)$ (cf. e.g. \cite{Monk2003}). Therefore, assuming that $\p\Om=\cup_{j=1}^{m}\Gamma_{j}$ and $S(\bold{x})=0$ on each connected piece $\Gamma_j$, it follows from \eqref{4th-bc_aux1}, \eqref{4th-bc_aux2} and \eqref{4th-bc_equiv}
in Lemma \ref{lem:4thBCequiv} that
\[
\begin{aligned}
(-\Delta^\ast\bold{u},\bold{v})
&=(\mu\bold{\nabla}\t\bold{\nabla}\t\bold{u},\bold{v})-((\lam+2\mu)\bold{\nabla}(\bold{\nabla}\cdot\bold{u}),\bold{v})\\
&=(\mu\bold{\nabla}\t\bold{u},\bold{\nabla}\t\bold{v})+((\lam+2\mu)\bold{\nabla}\cdot\bold{u},\bold{\nabla}\cdot\bold{v}),
\end{aligned}
\]
which, together with \eqref{lame-sys}, implies the solution $\bold{u}$ satisfies
\[
    (\mu\bold{\nabla}\t\bold{u},\bold{\nabla}\t\bold{v})+((\lam+2\mu)\bold{\nabla}\cdot\bold{u},\bold{\nabla}\cdot\bold{v})-\om^2(\bold{u},\bold{v})=(\bold{f},\bold{v})\quad\forall~\bold{v}\in\bold{X}_N.
\]

Now the variational formulation of \eqref{lame-sys} with the fourth kind boundary condition \eqref{4th-bc}, provided that $\mathcal{S}(\bold{x})=0$ on $\p\Om$, is to seek $\bold{u}=\bold{u}([\Om,\textrm{IV}])\in\bold{X}_N$ such that
\begin{equation}\label{vp_4th-bc}
a_4(\bold{u},\bold{v})=(\bold{f},\bold{v})\quad\forall~\bold{v}\in\bold{X}_N,
\end{equation}
where for any $\bold{v},\bold{w}\in \bold{X}_N$, the bilinear form $a_4(\cdot,\cdot)$ is defined as
\begin{equation}
a_4(\bold{v},\bold{w}):=(\mu\bold{\nabla}\t\bold{v},\bold{\nabla}\t\bold{w})+((\lam+2\mu)\bold{\nabla}\cdot\bold{v},\bold{\nabla}\cdot\bold{w})-\om^2(\bold{v},\bold{w}).
\end{equation}



\begin{thm}\label{thm:4thweak}
	For all $\omega$ but excluding a countable set of value, given any $\bold{f}\in \bold{L}^2(\Omega)$, there exists a unique $\bold{u}\in \bold{X}_N$ satisfies \eqref{vp_4th-bc}.
	Moreover,
	\begin{equation}\label{vp_4th-bc_est}
	\|\bold{u}\|_{\bold{X}_N}\le C\|\bold{f}\|_{\bold{L}^2(\Omega)},
	\end{equation}
	with some constant $C$ independent of $\bold{f}$.
\end{thm}
\begin{rem}
 Theorem~\ref{thm:4thweak} implies that \eqref{vp_4th-bc} is a well-posed forward problem even if there is no geometric condition $\mathcal{S}(\bold{x})=0$ on $\p\Om$.
\end{rem}

\begin{proof}
    For any $\bold{w}\in\bold{L}^2(\Om)$, we define an operator $\bold{K}_N:\bold{L}^2(\Om)\to \bold{X}_N$ by
    \[
        a_4(\bold{K}_N\bold{w},\bold{v})+(1+\omega^2)(\bold{K}_N\bold{w},\bold{v})
    =(\bold{w},\bold{v}),\quad \forall \bold{v}\in \bold{X}_N.
    \]
    Since it is easy to check that $a_4(\cdot,\cdot)+(1+\omega^2)(\cdot,\cdot)$ is bounded and coercive on $\bold{X}_N$ with respect to the graph norm, by the Lax-Milgram lemma $\bold{K}_N\bold{w}$ is well-defined and satisfies
    \[
    \|\bold{K}_N\bold{w}\|_{\bold{X}}\le C_1 \|\bold{w}\|_{\bold{L}^2(\Omega)}.
    \]
    Hence $\bold{K}_N:\bold{L}^2(\Om)\to \bold{X}_N$ is linear and bounded. This allows us to further rewrite \eqref{vp_4th-bc} as the problem of finding $\bold{u}\in\bold{X}_N$ such that
     \[
        \bold{u}-(\omega^2+1)\bold{K}_N\bold{u}=\bold{K}_N\bold{f}.
     \]
    Since $\bold{X}_N$ is compactly embedded into $\bold{L}^2(\Om)$, $\bold{K}_N:\bold{X}_N\to\bold{X}_N$ is compact. The Fredholm alternative theory implies the unique solvability of \eqref{vp_4th-bc} and the estimate \eqref{vp_4th-bc_est}.
\end{proof}

The following theorem retrieves the boundary condition in \eqref{4th-bc_equiv} from the variational formulation \eqref{vp_4th-bc}

\begin{thm}\label{thm:4thweakequ}
    Suppose that the solution $\bold{u}=\bold{u}[\Om;IV]$ to \eqref{vp_4th-bc} is in $\bold{H}^2(\Om)$. Then it solves \eqref{lame-sys} with
    \begin{equation}\label{4th-bc_aux}
        \bold{\nu}\t\bold{u}=\bold{0},\quad\bold{\nabla}\cdot\bold{u}=0\quad\mbox{on}~\p\Om.
    \end{equation}
\end{thm}

\begin{proof}
    By integration by parts and \eqref{atwo}, it is easy to check that the solution to \eqref{vp_4th-bc} also solves \eqref{lame-sys}. The first boundary condition in \eqref{4th-bc_aux} follows directly from the definition of $\bold{X}_N$. On any regular piece $\Gamma$ of $\p\Om$, for any $\phi\in C^{\infty}_c(\Gamma)$, by the trace theorem for Sobolev spaces \cite{Monk2003} there exists a $\bold{v}\in \bold{H}^1(\Om)$ such that
    \[
        \bold{\nu}\t\bold{v}=\bold{0}\quad\mbox{on}~\p\Om,
        \quad
            \bold{v}\cdot\bold{\nu}=\left\{\begin{array}{ll}
            \phi\quad\mbox{on}~\Gamma,\\
            0\quad\mbox{on the rest of}~{\color{black}\p\Om}.
                \end{array}
        \right.
        \]
    It is clear such a $\bold{v}\in \bold{X}_N$. Substituting it into \eqref{vp_4th-bc}, integrating by parts and noting \eqref{lame-sys}, we obtain
    \[
        ((\lam+2\mu)\bold{\nabla}\cdot\bold{u},\phi)_{\Gamma}=0\quad\forall~\phi\in C^{\infty}_c(\Gamma).
    \]
    As is $\Gamma$ is arbitrary, we obtain the second equation in \eqref{4th-bc_aux}.
\end{proof}

\begin{rem}\label{rem:vp-4th-bc2}
If $\mathcal{S}(\bold{x})=0$ is further imposed on $\p\Om$, then the two boundary conditions in \eqref{4th-bc_aux}, together with the arguments in the proof of \cite[Corollary 2.1]{liuxiao}, lead to the second boundary condition in \eqref{4th-bc}. This indicates that \eqref{vp_4th-bc} is equivalent to \eqref{lame-sys} and \eqref{4th-bc} if $\mathcal{S}(\bold{x})=0$ on $\p\Om$.
\end{rem}

We are now in a position to derive the following decoupling result based on the weak formulation \eqref{vp_4th-bc}.

\begin{thm}\label{thm:4thdecoup}
	Suppose the solution $\bold{u}=\bold{u}([\Om,\textrm{IV}])$ to \eqref{vp_4th-bc} is in $\bold{H}^3(\Om)$ and $\bold{f}\in\bold{H}^1(\Om)$.
	If we set $v_p=v_p([\Om,\textrm{IV}])=:-\bold{\nabla}\cdot\bold{u}$, then $v_p$ solves
	\begin{equation}\label{hel_4th-bc}
	\left\{\begin{array}{ll}
	\Delta v_p+k_p^2v_p=\frac{1}{\lam+2\mu}\bold{\nabla}\cdot\bold{f}\quad\m{in}~\Om,\\
	v_p=0\quad\m{on}~\p\Om.
	\end{array}
	\right.
	\end{equation}
	If we set $\bold{E}_s=\bold{E}_s([\Om,\textrm{IV}]):=\bold{\nabla}\t\bold{u}$, then $\bold{E}_s$ solves the following Maxwell system:
	\bb\label{max_4th-bc}
	\left\{\begin{array}{ll} \bold{\nabla}\t(\bold{\nabla}\t\bold{E}_s)-k_s^2\bold{E}_s=\frac{1}{\mu}\bold{\nabla}\t\bold{f}\quad\m{in}~\Om,\\
        \bold{\nabla}\cdot\bold{E}_s=0\quad\m{in}~\Om,\\
		\bold{\nu}\t(\bold{\nabla}\t\bold{E}_s)=
		\frac{1}{\mu}\bold{\nu}\t\bold{f}\quad\m{on}~\p\Om.
	\end{array}\right.
	\ee
\end{thm}
\begin{proof}
	For any $\phi\in C_0^\infty(\Om)$, integration by parts yields
	\begin{align*}
	a_4(\bold{u},\bold{\nabla}\phi)
	&=
	((\lam+2\mu)\bold{\nabla}\cdot\bold{u},\bold{\nabla}\cdot\bold{\nabla}\phi)-\om^2(\bold{u},\bold{\nabla}\phi)
	\\&=-((\lam+2\mu)\bold{\nabla}(\bold{\nabla}\cdot\bold{u}),\bold{\nabla}\phi)+\om^2(\bold{\nabla}\cdot\bold{u},\phi).
	\end{align*}
	In addition, we have
	\[
	(\bold{f},\bold{\nabla}\phi)=-(\bold{\nabla}\cdot\bold{f},\phi).
	\]
	Therefore from \eqref{vp_4th-bc} and the second equation in \eqref{4th-bc_aux} we arrive at the Helmholtz system \eqref{hel_4th-bc} with $v_p=-\bold{\nabla}\cdot\bold{u}$. Next we turn our attention to \eqref{max_4th-bc}. Taking $\bold{v}=\bold{\nabla}\t\bold{F}$ for any $\bold{F}\in \bold{C}^\infty_0(\Om)$ in \eqref{vp_4th-bc} leads to
	\begin{align*}
	a_4(\bold{u},\bold{\nabla}\t\bold{F})
	=&
	(\mu\bold{\nabla}\t\bold{u},\bold{\nabla}\t\bold{\nabla}\t\bold{F})-\om^2(\bold{u},\bold{\nabla}\t\bold{F})
	\\=&(\mu\bold{\nabla}\t\bold{\nabla}\t(\bold{\nabla}\t\bold{u}),\bold{F})-\om^2(\bold{\nabla}\t\bold{u},\bold{F}),
	\end{align*}
	and
	\[
	(\bold{f},\bold{\nabla}\t\bold{F})=(\nabla\t\bold{f},\bold{F}).
	\]
	Hence we obtain the first equation and the second equation in \eqref{max_4th-bc} with $\bold{E}_s=\bold{\nabla}\t\bold{u}$. It remains to prove the third boundary condition. From \eqref{thm:4thweakequ}, we know $\bold{u}$ solves \eqref{lame-sys}. This implies
\[
    \mu\bold{\nu}\t\bold{\nabla}\t(\bold{\nabla}\t\bold{u})-
    (\lambda+2\mu)\bold{\nu}\t\bold{\nabla}(\bold{\nabla}\cdot\bold{u})-\om^2\bold{\nu}\t\bold{u}=\bold{\nu}\t\bold{f}\quad\mbox{on}~\p\Om.
\]
Noting \eqref{4th-bc_aux} in Theorem \ref{thm:4thweakequ}, we come to the conclusion.
\end{proof}

\begin{rem}\label{rem:vp-4th-bc3}
    In the proof of Theorem \ref{thm:4thweak}, we also derive a boundary condition
    \[
        \bold{\nu}\t\bold{\nabla}\t(\bold{\nabla}\t\bold{u})=\frac{1}{\mu}\bold{\nu}\t\bold{f}\quad\mbox{on}~\p\Om.
    \]
    This and the second boundary condition in \eqref{4th-bc_aux} have
    been deduced from \eqref{lame-sys} and \eqref{4th-bc} with
    $\bold{f}=\bold{0}$ under the assumption that $\mathcal{S}(\bold {x})=0$ on $\p\Om$ in \cite{liuxiao} and then the decoupling result \eqref{hel_4th-bc} and \eqref{max_4th-bc} is established while the geometric assumption is not used here; also see the proof of Theorem \ref{thm:4thweakequ}. This is not a conflict because the variational formulation \eqref{vp_4th-bc} is derived based on the {\color{black}geometric assumption.} 
    In other words, \eqref{vp_4th-bc} is independent of $\mathcal{S}(\bold{x})$ due to the geometric assumption $\mathcal{S}(\bold{x})=0$ on $\p\Om$.
\end{rem}


\subsection{The third kind boundary condition}
In this subsection, we treat \eqref{lame-sys} with the third kind boundary condition \eqref{3rd-bc}. As before, we recall a result for the third kind boundary condition \eqref{3rd-bc} built in \cite{liuxiao}.
\begin{lem}[\cite{liuxiao}]\label{lem:3rdBCequiv}
	Let $\bold{u}$ solve \eqref{lame-sys} with  $\bold{f}\in\bold{L}^2(\Om)$. If a connected piece $\Gamma$ of $\p\Om$ is flat, then the third kind boundary condition \eqref{3rd-bc} implies
	\bb\label{3rd-bc_equiv}
	\bold{\nu}\t(\bold{\nabla}\t\bold{u})=\bold{0}\quad\m{on}~\Gamma,
	\ee
\end{lem}

Suppose $\bold{u}\in \bold{H}^2(\Om)$ solves \eqref{lame-sys} and \eqref{3rd-bc}. Arguing as before, we use integration by parts to get for any $\bold{v}\in\bold{X}_T$
\begin{align*}
    -(\Delta^\ast\bold{u},\bold{v})=&(\mu\bold{\nabla}\t\bold{\nabla}\t\bold{u},\bold{v})-((\lam+2\mu)\bold{\nabla}(\bold{\nabla}\cdot\bold{u}),\bold{v})\nb\\
    =&(\mu\bold{\nabla}\t\bold{u},\bold{\nabla}\t\bold{v})+((\lam+2\mu)\bold{\nabla}\cdot\bold{u},\bold{\nabla}\cdot\bold{v})
    \\&+\langle\mu\bold{\nu}\t(\bold{\nabla}\t\bold{u}),\bold{v}\rangle_{\p\Om}\nb.
\end{align*}
Further assuming $\Om$ is a Lipschitz polyhedron, this, \eqref{3rd-bc_equiv} in Lemma \ref{lem:3rdBCequiv} and \eqref{lame-sys} tell that $\bold{u}$ also solves
\[
    (\mu\bold{\nabla}\t\bold{u},\bold{\nabla}\t\bold{v})+((\lam+2\mu)\bold{\nabla}\cdot\bold{u},\bold{\nabla}\cdot\bold{v})-\om^2(\bold{u},\bold{v})=(\bold{f},\bold{v})\quad\forall~\bold{v}\in\bold{X}_T.
\]
%
For any $\bold{v},\bold{w}\in \bold{X}_T$, the bilinear form $a_3(\cdot,\cdot)$ is defined as
\begin{equation}
a_3(\bold{v},\bold{w}):=(\mu\bold{\nabla}\t\bold{v},\bold{\nabla}\t\bold{w})+((\lam+2\mu)\bold{\nabla}\cdot\bold{v},\bold{\nabla}\cdot\bold{w})-\om^2(\bold{v},\bold{w}).
\end{equation}
Thus, if $\Om$ is a Lipschitz polyhedron the boundary value problem \eqref{lame-sys} with the third kind boundary condition \eqref{3rd-bc} can be expressed as finding $\bold{u}=\bold{u}[\Om;III]\in\bold{X}_T$ such that
\begin{equation}\label{vp_3rd-bc}
a_3(\bold{u},\bold{v})=(\bold{f},\bold{v})\quad\forall~\bold{v}\in\bold{X}_T.
\end{equation}

\begin{thm}\label{thm:3rdweak}
	For all $\omega$ but excluding a countable set of value, given any $\bold{f}\in \bold{L}^2(\Omega)$, there exists a unique $\bold{u}\in \bold{X}_T$ satisfies \eqref{vp_3rd-bc}.
	Moreover,
	\begin{equation}\label{vp_3rd-bc_est}
	\|\bold{u}\|_{\bold{X}_T}\le C\|\bold{f}\|_{\bold{L}^2(\Omega)},
	\end{equation}
	with some constant $C$ independent of $\bold{f}$.
\end{thm}
\begin{rem}
As in the case of the fourth kind boundary condition, Theorem \ref{thm:3rdweak} implies \eqref{vp_3rd-bc} is a well-posed problem regardless of whether $\Om$ is a Lipschitz polyhedron.
\end{rem}

\begin{proof}
    As in the proof of Theorem \ref{thm:4thweak},  \eqref{vp_3rd-bc} is reformulated as an operator equation
\[
    \bold{u}-(\om^2+1)\bold{K}_T\bold{u}=\bold{K}_T\bold{f},
\]
with $\bold{K}_T:\bold{L}^2(\Om)\to\bold{X}_T$ defined by
 \[
a_3(\bold{K}_T\bold{w},\bold{v})+(1+\omega^2)(\bold{K}_T\bold{w},\bold{v})
=(\bold{w},\bold{v}),\quad \forall \bold{v}\in \bold{X}_T.
\]
The compact embedding of $\bold{X}_T$ into $\bold{L}^2(\Om)$ (cf. Lemma \ref{compact}) and the Fredholm alternative theory ensures the well-posedness of \eqref{vp_3rd-bc} and \eqref{vp_3rd-bc_est}. 
\end{proof}

\begin{thm}\label{thm:3rdweakequ}
    Suppose that the solution $\bold{u}=\bold{u}[\Om;III]$ to \eqref{vp_3rd-bc} is in $\bold{H}^2(\Om)$. Then it solves \eqref{lame-sys} with
    \begin{equation}\label{3rd-bc_aux}
        \bold{\nu}\cdot\bold{u}=0,\quad\bold{\nu}\t(\bold{\nabla}\t\bold{u})=\bold{0},
\quad\mbox{on}~\p\Om.
    \end{equation}
\end{thm}
\begin{proof}
    It is straightforward to get \eqref{lame-sys} from \eqref{vp_3rd-bc} by integration by parts. Since $\bold{u}\in\bold{X}_T$, it has a vanishing normal trace on $\p\Om$, i.e., the first boundary condition in \eqref{3rd-bc_aux}. It remains to prove the second boundary condition in \eqref{3rd-bc_aux}. On any regular piece $\Gamma$ of $\p\Om$, for any tangential vector field $\bold{\phi}\in (C^{\infty}_c(\Gamma))^2$, the trace theorem for Sobolev spaces implies that there exists a $\bold{v}\in \bold{H}^1(\Om)$ such that
    \[
        \bold{\nu}\t\bold{v}=\left\{\begin{array}{ll}
            \bold{\phi}\quad\mbox{on}~\Gamma,\\
            0\quad\mbox{on the rest of}~\p\Om,
                \end{array}
        \right.
        \quad
            \bold{v}\cdot\bold{\nu}=\bold{0}\quad\mbox{on}~\p\Om,
        \]
    Such a $\bold{v}$ is clearly in $\bold{X}_T$. We insert it in \eqref{vp_3rd-bc}, perform integration by parts and use \eqref{lame-sys} to get for any tangential vector field $\bold{\phi}\in (C_c^\infty(\Gamma))^2$
    \[
        -(\mu\bold{\nu}\t(\bold{\nabla}\t\bold{u}),\bold{v})_{\p\Om}=(\mu\bold{\nu}\t(\bold{\nabla}\t\bold{u}),\bold{\nu}\t\bold{\phi})_{\Gamma}=0.
    \]
    The second boundary condition follows as a result of the arbitrariness of $\Gamma$.
\end{proof}


\begin{rem}\label{rem:vp-3rd-bc2}
    As in the case of the fourth boundary condition, if $\Om$ is further assumed to be a Lipschitz polyhedron, we may argue as in the proof of \cite[Corollary 2.2]{liuxiao} to deduce the second condition in \eqref{3rd-bc} from \eqref{3rd-bc_aux}. Therefore, the variational problem \eqref{vp_3rd-bc} and the Lam\'{e} system \eqref{lame-sys} with \eqref{3rd-bc} are equivalent when $\Om$ is a Lipschitz polyhedron.
\end{rem}

Analogous to Theorem~\ref{thm:4thdecoup}, we have the following decoupling result for the third kind boundary condition.
\begin{thm}\label{thm:3rddecoup}
Suppose the solution $\bold{u}=\bold{u}([\Om,\textrm{III}])$ to \eqref{vp_3rd-bc} is in $\bold{H}^3(\Om)$ and $\bold{f}\in\bold{H}^1(\Om)$.
	Then
	\begin{equation}\label{hel_3rd-bc}
	\left\{\begin{array}{ll}
	\Delta v_p+k_p^2v_p=\frac{1}{\lam+2\mu}\bold{\nabla}\cdot\bold{f}\quad\m{in}~\Om,\\
	\p_{\bold{\nu}} v_p=\frac{1}{\lam+2\mu}\bold{\nu}\cdot\bold{f}\quad\m{on}~\p\Om,
	\end{array}
	\right.
	\end{equation}
	and
	\begin{equation}\label{max_3rd-bc}
	\left\{\begin{array}{ll}
	\bold{\nabla}\t(\bold{\nabla\t\bold{E}_s})-k_s^2\bold{E}_s=\frac{1}{\mu}\bold{\nabla}\t\bold{f}\quad\m{in}~\Om,\\
\bold{\nabla}\cdot\bold{E}_s=0\quad\m{in}~\Om,\\
	\bold{\nu}\t\bold{E}_s=\bold{0}\quad\m{on}~\p\Om,
	\end{array}\right.
	\end{equation}
	 with $v_p=v_p([\Om,\textrm{III}])=:-\bold{\nabla}\cdot\bold{u}$ and $\bold{E}_s=\bold{E}_s([\Om,\textrm{III}]):=\bold{\nabla}\t\bold{u}$.
\end{thm}
\begin{proof}
	The proof is similar to that of Theorem \ref{thm:4thdecoup}. We start with the Maxwell system \eqref{max_3rd-bc}. Taking $\bold{v}=\bold{\nabla}\t\bold{F}$ with any $\bold{F}\in \bold{C}_0^\infty(\Om)$ in \eqref{vp_3rd-bc} and integration by parts give
\[
    (\mu\bold{\nabla}\t\bold{\nabla}\t(\bold{\nabla}\t\bold{u}),\bold{F})-\omega^2(\bold{\nabla}\t\bold{u},\bold{F})=(\bold{\nabla}\t\bold{f},\bold{F}).
\]
From this and the second equation in \eqref{3rd-bc_aux}, we get \eqref{max_3rd-bc} with $\bold{E}_s=\bold{\nabla}\t\bold{u}$. Likewise, we set $v=\bold{\nabla}\phi$ in \eqref{vp_3rd-bc} for any $\phi\in C_0^\infty(\Om)$. Then integration by parts yields the first equation in \eqref{hel_3rd-bc} with $v_p=-\bold{\nabla}\cdot\bold{u}$. Next we show the second boundary condition in \eqref{hel_3rd-bc}. In view of Theorem \ref{thm:3rdweakequ}, $\bold{u}$ also solves the Lam\'{e} system \eqref{lame-sys}. For any $\phi\in H^1(\Om)$, we multiply both sides of \eqref{lame-sys} by $\bold{\nabla}\phi$ and perform integration by parts with the help of \eqref{atwo} to get
\begin{align*}
    &\quad   (-\Delta^\ast\bold{u}-\omega^2\bold{u},\bold{\nabla}\phi)\quad\forall~\phi\in H^1(\Om)\\
    &=(\mu\bold{\nabla}\t\bold{\nabla}\t\bold{u},\bold{\nabla}\phi)-((\lam+2\mu)\bold{\nabla}(\bold{\nabla}\cdot\bold{u}),\bold{\nabla}\phi)-\om^2(\bold{u},\bold{\nabla}\phi)\\
    &\underbrace{=}_{\eqref{3rd-bc_aux}}((\lam+2\mu)\Delta\bold{\nabla}\cdot\bold{u},\phi)-((\lambda+2\mu)\bold{\nabla}(\bold{\nabla}\cdot\bold{u})\cdot\bold{\nu},\phi)_{\p\Om}+\om^2(\bold{\nabla}\cdot\bold{u},\phi),
\end{align*}
    \[
        (\bold{f},\bold{\nabla}\phi)=-(\bold{\nabla}\cdot\bold{f},\phi)+(\bold{\nu}\cdot\bold{f},\phi)_{\p\Om}\quad\forall~\phi\in H^1(\Om).
    \]
    Finally the desired result is concluded from \eqref{lame-sys} and the first equation in \eqref{hel_3rd-bc}.
\end{proof}

\begin{rem}
    In the above proof, we in fact get another boundary condition
    \[
        \bold{\nabla}\cdot\bold{\nabla}(\bold{\nabla}\cdot\bold{u})=-\frac{1}{\lam+2\mu}\bold{\nabla}\cdot\bold{f}\quad\mbox{on}~\p\Om.
    \]
    In \cite{liuxiao}, this and the second condition in \eqref{3rd-bc_aux} have been obtained for the homogeneous Lam\'{e} system coupled with \eqref{3rd-bc} under the assumption that $\Om$ is a Lipschitz polyhedron, which then ensures \eqref{hel_3rd-bc} and \eqref{max_3rd-bc}. Like the case of the fourth boundary condition, the reason we retrieve these results here is that \eqref{vp_3rd-bc} is posed over a Lipschitz polyhedron $\Om$.
\end{rem}

\section{Concluding Remark}

In this paper, two weak formulations are proposed for the Lam\'{e} system with the third and the fourth boundary condition, respectively. Under some geometric conditions on the boundary surface and a reason regularity assumption on the solution, we prove the equivalence of these two formulations. Moreover, the decoupling results in \cite{liuxiao} are concluded from the weak formulations if $\mathcal{S}(\bold{x})=0$ on $\Om$ or $\Om$ is a Lipschitz polyhedron. One may naturally raise a question whether above results are true when the geometric assumptions do not hold. This causes so-called imperfect decoupling of the Lam\'{e} system. We shall explore this problem in a forthcoming paper.

\section*{Acknowledgement}

The work of J Lai was supported by Natural Science Foundation of Fujian Province of China (2016J01670).
The work of H Liu was supported by the startup fund and FRG grants from Hong Kong Baptist University and the Hong Kong RGC grants (projects 12302415 and 12302017). The work of Y Xu was supported by National Natural Science Foundation of China (11201307), Ministry of Education of China through Special Research Fund for the Doctoral Program of Higher Education (20123127120001) and Natural Science Foundation of Shanghai (17ZR1420800).

\end{document}